\documentclass[12pt,letterpaper]{amsart}

\usepackage{mathrsfs}
\usepackage{latexsym}
\usepackage{amsmath,amsfonts,amsthm,amssymb,amscd}
\input amssym.def
\input amssym.tex

\newtheorem{theorem}{Theorem}[section]
\newtheorem*{theorem*}{Theorem}

\newtheorem{lemma}[theorem]{Lemma}

\theoremstyle{remark}

\newcommand{\ud}{\mathrm{d}}
\newcommand{\dx}{\mathrm{d}x}

\newcommand{\abs}[1]{\left|#1\right|}

\newcommand{\floor}[1]{\left\lfloor#1\right\rfloor}

\newcommand{\printlineno}{\hspace*{-.5in}\texttt{SOURCE LINE \#\the\inputlineno}}
\renewcommand{\printlineno}{}
\printlineno

\newcommand{\psage}{{PSAGE}}
\newcommand{\sage}{{Sage}}
\newcommand{\pari}{{PARI}}

\title{Conditionally bounding analytic ranks of elliptic curves}
\author{Jonathan W. Bober}
\address{Department of Mathematics, University of Washington, Seattle, WA, USA}
\email{jwbober@math.washington.edu}
\date{December 7, 2011}

\begin{document}

\begin{abstract}
We describe a method for bounding the rank of an elliptic curve under the assumptions of
the Birch and Swinnerton-Dyer conjecture and the generalized Riemann hypothesis. As an example, we
compute, under these conjectures, exact upper bounds for curves which are known to have
rank at least as large as $20, 21, 22, 23$, and $24$. For the known curve of rank at
least $28$, we get a bound of $30$.
\end{abstract}
\maketitle

\section{Introduction}

Determining the rank of an elliptic curve is a difficult problem, and there is currently
no known unconditional algorithm for determining the rank of a given curve. The basic
method for rigorously determining the rank of a curve is to find an upper bound for
the rank by computing the size of some Selmer groups and to find a lower bound for
the rank by finding enough independent rational points. In theory, if one continues
this process long enough, and the Shafarevich-Tate group of the curve is finite, the
upper and lower bounds should eventually coincide and the rank will be determined
exactly.

In practice, things are not so simple. Finding points on the curve is sometimes not too bad,
but the upper bounds for the rank are more problematic. Even the computation of the $2$-Selmer
rank is difficult, and it becomes prohibitively time consuming as the coefficients of
the elliptic curve grow; it is easy to write down a curve for which the state of the art
program for computing the $2$-Selmer group, John Cremona's mwrank \cite{mwrank},
will effectively take ``forever.''

If one is willing to accept the Birch and Swinnerton-Dyer conjecture that the rank of an
elliptic curve is the same as the order of vanishing of its $L$-function at the central
point, then it is possible to use the $L$-function to get information about the
rank. In fact, when the order of vanishing is between $0$ and $3$, it can be possible
to compute the $L$-function to enough precision and use some extra information about
the curve to determine the analytic rank exactly, as is done in \cite{rank3-bsd}, for example.
When the rank is larger than this, though, currently the best one can do is determine
that the first $r$ derivatives of the $L$-function are very close to $0$, and the
$(r+1)$-st is not, which will provide a very good guess for the rank and a rigorous upper
bound, assuming BSD.

This approach has its own problems, as it is much easier to write down a curve
of large conductor than it is to compute the $L$-function of such a curve.
For example, the known curve of rank at least $28$ \cite{rank28-curve}, which we will
write down later,
has conductor $N \approx 3.5 \times 10^{141}$,
and current methods (such as those described in \cite{rubinstein-methods-and-experiments})
typically
require summing on the order of $\sqrt{N}$ terms to compute the central value of the $L$-function. 
(It would take a compute about $10^{53}$ cpu-years just to add $1$ to itself $10^{70}$ times.)

We present here a third method which is rather effective at bounding the rank, especially
when the rank is large compared to the conductor, as long as one is willing
to assume both the Birch and Swinnerton-Dyer
conjecture and the Riemann Hypothesis for the $L$-function of the curve.
This method is not completely new. It is based on
Mestre's method \cite{mestre} for (conditionally) bounding the rank of an elliptic curve based
only on its conductor, and it was used by Fermigier \cite{fermigier} to study ranks of
elliptic curves in certain families. However, it does not seem to have gained much traction
and does not seem to have been used much, if at all, since.

The idea, in brief, is as follows. Take $f(x)$ to be a function such that
$f(0) = 1$ and $f(x) \ge 0$ for all real $x$. Then, assuming the Riemann hypothesis,
the sum $\sum f(\gamma)$, where $1/2 + i\gamma$ runs over the nontrivial zeros of $L(s, E)$,
will be an upper bound for the analytic rank of $E$. Moreover,
for certain choices of $f(x)$ this sum may be efficiently evaluated using the
explicit formula for the $L$-function attached to $E$.


This method has recently been implemented by the author, and is available as part
of William Stein's \psage{ }\cite{psage} add-ons to \sage{ }\cite{sage-4.7.2}. As an example,
of what it can do, we will examine $6$ curves that are known to have rather large rank. We denote
these curves as $E_n$, $n = 20, 21, 22, 23, 24, 28$, where $n$ is a known lower bound for
the rank. We will write down these curves later (they are all taken from A. Dujella's
website \cite{rank-records}, and
at the time of discovery each held the record for the curve with largest number of known
independent rational points). The exact rank is not known for any of these curves. However,
conditionally we may claim

\begin{theorem}\label{main-theorem}
Assuming BSD and GRH, $E_n$ has rank exactly $n$ for $n = 20, 21, 22, 23,$ and $24$, while
$E_{28}$ has rank $28$ or $30$.
\end{theorem}

\subsection{Acknowledgements} 
Most of the computations in this paper run in a short amount of time,
and were done on the author's personal computer. Some longer computations
were run on the sage cluster at the University of Washington, supported
by NSF grant DMS-0821725, and the riemann cluster at the University
of Waterloo, funded by the Canada Foundation for Innovation, the Ontario
Innovation Trust, and SGI.

The source code for our implementation is available as part of PSAGE \cite{psage}.
It uses Sage \cite{sage-4.7.2}, and hence PARI \cite{pari-2.4.3}, to compute $a_p$ for
bad primes, and uses Andrew Sutherland's smalljac \cite{smalljac} to compute
all other values of $a_p$.

Parts of this work began while the author was in residence at the
Mathematical Sciences Research Institute during the Arithmetic
Statistics program, Spring 2011, during which time the author was
partially supported by NSF grant DMS-0441170, administered by MSRI.
Discussions during the informal ``explicit formula seminar,'' especially
with David Farmer and Michael Rubinstein, were influential in encouraging 
this work.

Currently the author is supported by NSF grant DMS-0757627, administered
by the American Institute of Mathematics.

\section{Bounding ranks}
\subsection{The method}
Let
\[
    L(s, E) = \sum_{n=1}^\infty \frac{a_n}{n^s} = \prod_p L_p(s, E)^{-1}
\]
be the $L$-function of an elliptic curve, normalized so that the completed $L$-function
$\Lambda(s, E) = \epsilon \Lambda(1 - s, E)$, and let $c_n$ be defined by
\[
    -\frac{L'(s, E)}{L(s, E)} = \sum_{n=1}^\infty \frac{c_n}{n^s}.
\]
More explicitly, if we define $\alpha(p)$ and $\beta(p)$ by
\[
    L_p(s, E) = (1 - \alpha(p)p^{-s})(1 - \beta(p)p^{-s}),
\]
(note that $\alpha$ and $\beta$ are only well defined up to permutation, and that at least one
of them will be $0$ when $p$ is a prime of bad reduction), then
\begin{equation}
    c_{p^m} = \big(\alpha(p)^m + \beta(p)^m\big)\log p,
\end{equation}
and $c_n = 0$ when $n$ is not a prime power.

Our main tool will be the explicit formula for $L(s, E)$, which we state in a
friendly form in the following lemma.

\begin{lemma}
Suppose that $f(z)$ is an entire function with $f(x + iy) \ll x^{-(1 + \delta)}$ for $\abs{y} < 1 + \epsilon$,
for some $\epsilon > 0$, and that the Fourier transform of $f$
\[
    \hat f(y) = \int_{-\infty}^\infty f(x)e^{-2 \pi i x y} \dx
\]
exists and is such that
\[
    \sum_{n=1}^\infty \frac{c_n}{n^{1/2}} \hat f\left(\frac{\log n}{2\pi}\right) 
\]
converges absolutely. Then
\begin{multline}\label{eq-xxx-general}
    \sum_{\gamma} f(\gamma) = \hat f(0) \frac{\log N}{2\pi} - \hat f(0) \frac{\log 2\pi}{\pi}
         + \frac{1}{\pi} \Re \left\{\int_{-\infty}^\infty \frac{\Gamma'}{\Gamma}(1 + it)f(t) \ud t\right\} \\
         - \frac{1}{2\pi}\sum_{n = 1}^\infty \frac{c(n)}{n^{1/2}}\left(\hat f\left(\frac{\log n}{2\pi}\right)
                + \hat f\left(-\frac{\log n}{2\pi}\right) \right),
\end{multline}
where $1/2 + i\gamma$ runs over the nontrivial zeros of $L(s, E)$, where $E$ is an elliptic
curve with conductor $N$.
\end{lemma}
\begin{proof}
A proof of the explicit formula in this form, or in a similar form, can be found in various
sources, e.g. \cite[Theorem 5.12]{iwaniec-kowalski}, so we give only a brief sketch. The
idea is to integrate the function
\[
    F(s) \frac{L'(s, E)}{L(s, E)},
\]
where $F(1/2 + is) = f(s)$, on a vertical line to the right of the critical strip and, in the
reverse direction, on a vertical line to the left of the critical strip. By the residue
theorem, this integral will be equal to $2\pi \sum_\gamma f(\gamma)$. One now applies the
functional equation to write the integral in the left half-plane as an integral in the right
half-plane.

The sum over the Fourier coefficients of $f$ arises from shifting contours to the region
of absolute convergence and using the Dirichlet series for $L'(s)/L(s)$, while the other
terms arise from shifting the remaining integrals to the line $\Re(s) = 1/2$.

The conditions on $f(z)$ are exactly those needed to make sure that this process can go
through without trouble. Of course, it is also important that $L(s, E)$ is entire
and that it satisfies a functional equation \cite{W, TW, BCDT}.
\end{proof}

A convenient function to use in an application of the explicit formula is
\[
    f(z) = f(z; \Delta) = \left(\frac{\sin(\Delta \pi z)}{\Delta \pi z}\right)^2,
\]
which has the simple Fourier transform
\[
    \hat f(x; \Delta) = \left(\frac{1}{\Delta}\right)\left(1 - \abs{\frac{x}{\Delta}}\right), \abs{x} < \Delta.
\]
With this choice of $f$, equation \eqref{eq-xxx-general} takes the form
\begin{multline}\label{eq-xxx-specific}
    \sum_{\gamma} f(\gamma; \Delta) = \frac{\log N}{\Delta 2\pi} - \frac{\log 2\pi}{\Delta \pi}
         + \frac{1}{\pi} \Re\left\{\int_{-\infty}^\infty \frac{\Gamma'}{\Gamma}(1 + it)f(t; \Delta)\ud t\right\} \\
         - \frac{1}{\Delta \pi} \sum_{p \le \exp(2\pi\Delta)} \log p \sum_{k=1}^{\floor{2\pi\Delta/\log p}} \frac{k}{p^{k/2}}\big(\alpha(p)^k + \beta(p)^k\big)\left(1 - \frac{k\log p}{2\pi\Delta}\right).
\end{multline}

Since $f(\gamma; \Delta) \ge 0$ as long as $\gamma$ is real, and $f(0; \Delta) = 1$,
equation \eqref{eq-xxx-specific} will give an upper bound for the order
of vanishing of $L(s, E)$ at $s = 1/2$, as long as the Riemann Hypothesis holds
for $L(s, E)$. And if
$\Delta$ is not too large, we can quickly evaluate the right hand side of equation
\eqref{eq-xxx-specific} to calculate this upper bound. It is also worth
noting that, assuming RH,
\begin{multline*}
    -\lim_{\Delta \rightarrow \infty} \frac{1}{\Delta \pi} \sum_{p \le \exp(2\pi\Delta)} \log p \sum_{k=1}^{\floor{2\pi\Delta/\log p}} \frac{k}{p^{k/2}}\big(\alpha(p)^k + \beta(p)^k\big)\left(1 - \frac{k\log p}{2\pi\Delta}\right) \\
     = \mathrm{ord}_{s = 1/2} L(s, E)
\end{multline*}
so that, in principle, we should be able to get as good a bound for the rank as we like
through this method. However, as the length of the prime sum grows exponentially in $\Delta$,
this method quickly becomes infeasible once $\Delta$ gets a little large than $4$.

\subsection{Some curves}
As an example, we examine $6$ elliptic curves from Dujella's online tables. They are
\begin{multline*}
    E_{20}: y^2 + xy = x^3 - 431092980766333677958362095891166x \\
        + 5156283555366643659035652799871176909391533088196,
\end{multline*}
\begin{multline*}
E_{21}: y^2 + xy + y = x^3 + x^2 - 215843772422443922015169952702159835x \\
              - 19474361277787151947255961435459054151501792241320535,
\end{multline*}
\begin{multline*}
E_{22}: y^2 + xy + y = x^3 - 940299517776391362903023121165864x \\
              + 10707363070719743033425295515449274534651125011362,
\end{multline*}
\begin{multline*}
E_{23}: y^2 + xy + y = x^3 - 19252966408674012828065964616418441723x  \\
              + 32685500727716376257923347071452044295907443056345614006,
\end{multline*}
\begin{multline*}
E_{24}: y^2 + xy + y = x^3 - 120039822036992245303534619191166796374x \\
              + 504224992484910670010801799168082726759443756222911415116,
\end{multline*}
and
\begin{multline*}
E_{28}: y^2 + xy + y = x^3 - x^2 - 
    {20067762415575526585033208 \times 10^{30} \choose +\ 209338542750930230312178956502}x \\
    + {3448161179503055646703298569039072037485594 \times 10^{40} \choose
       +\ 4359319180361266008296291939448732243429}.
\end{multline*}

Each $E_n$ has $n$ known independent rational points of infinite order, so thus has at least
rank $n$.
(See \cite{rank20-curve, rank21-curve, rank22-curve, rank23-curve, rank24-curve, rank28-curve},
or \cite{rank-records} for quick reference.)
Using the methods described above, we compute rank
bounds for each of these curves. These are listed in Table \ref{table1}. The global root number
can be computed for each curve. (In \sage, {\tt E.root\_number()}, which
uses \pari{ }\cite{pari-2.4.3}, will finish quickly for
$E_{20}$, $E_{21}$, and $E_{22}$ and within a few hours for $E_{23}$ and $E_{24}$. For $E_{28}$
it is best to see the mailing list discussion which gives the factorization of the
discriminant \cite{rank28-discussion}.) In each case the root number agrees with the parity
of the known number of independent points, so to get a tight upper bound for the rank
we only need to get within $2$ of the number of known independent points, and so
the computation in Table \ref{table1} gives the proof of Theorem \ref{main-theorem}.


\begin{table}
\begin{tabular}{c|c|c|c|c}
Curve & $\log N_E$ &  $\Delta$ & $\sum_{\gamma} f(\gamma; \Delta)$ & $\frac{\log N_E}{2\pi\Delta}$\\\hline
$E_{20}$ & $170.09$ & $2.0$ & $21.70$ & $13.54$ \\
$E_{21}$ & $196.68$ & $2.5$ & $22.68$ & $12.52$ \\
$E_{22}$ & $182.72$ & $2.0$ & $23.71$ & $14.54$ \\
$E_{23}$ & $205.06$ & $2.5$ & $24.49$ & $13.05$ \\
$E_{24}$ & $219.93$ & $2.5$ & $25.57$ & $14.00$ \\
$E_{28}$ & $325.90$ & $3.2$ & $31.30$ & $16.21$ 
\end{tabular}

\vspace{.1in}

\caption{Computed upper bounds for the ranks of some curves, along with a heuristic
guess of what these bounds should for a typical elliptic curve. The sum over the zeros here is
rounded up; other numbers are rounded to nearest.} \label{table1}
\end{table}

\subsection{Curves of small conductor}\label{section-small-conductor}
For further testing, this method was also run
on all elliptic curves up with conductor below $180000$
(from Cremona's tables \cite{cremona-online})
using $\Delta = 2.0$, a computation
which ran in under a day on a fast $8$ core computer. In this range there are $790677$
isogeny classes of elliptic curves, and for all but $9882$ isogeny classes it
turns out that $\floor{\sum_{\gamma} f(\gamma; 2.0)} = \mathrm{rank}(E)$; in the remaining
cases, $\floor{\sum_{\gamma} f(\gamma; 2.0)} = \mathrm{rank}(E) + 1$, so consideration
of the root number of the curve gives the exact rank. 

\section{Further comments}

\subsection{Some evidence towards BSD}
There is a way in which these computations can be seen as giving mild evidence in support of the
Birch and Swinnerton-Dyer conjecture. The upper bound computed for a curve $E$ is the 
value of the sum $\sum_\gamma f(\gamma; \Delta)$, and as $f(\gamma; \Delta)$
decays fairly rapidly as $\gamma$ grows, one does not expect this sum to be very large
for a typical elliptic curve.

To obtain a crude approximation to what we might expect the value of this sum to be, consider
that the local zero density of a typical $L(s, E)$ near the central point is approximately
$\frac{2\pi}{\log N_E}$. Then, if the zeros are spaced uniformly at random (an assumption
that is not really correct, but is close enough to true for our crude purposes), we might expect
that
\[
    \sum_{\gamma} f(\gamma, \Delta) \approx \frac{\log N_E}{2\pi} \int_{-\infty}^\infty f(t; \Delta) \ud t
         = \frac{\log N_E}{2\pi\Delta},
\]
possibly with a small adjustment to take into account the parity of the rank. (More precisely,
we might expect that if we average this sum over all elliptic curves of conductor close to $N_E$,
the answer will not be too far from this integral.) Thus, when this sum is significantly larger
than this estimate, it indicates an extreme concentration of zeros near the central point. (It
is also possible to arrive at more refined version of this heuristic by considering the explicit
formula. In such a case, it is necessary to assume that the family of elliptic curves considered
is large enough that $a_p(E)$ averages to zero for each $p$, and we notice that the integral
of the $\Gamma$-factor plays a small role as well.)

As some further small evidence for this heuristic, we note that the average of
\[
    \frac{4 \pi}{\log N} \sum_{\gamma} f(\gamma; 2.0)
\]
over all isogeny classes up to $180000$ is approximately $.9638$. The small difference from
$1$ should be accounted for by the $\Gamma$-factor, which tends to push zeros away from the
central point.

It should also be possible to refine this heuristic somewhat to make a guess as to what
the sum should be for a high rank curve by making the assumption that a zero of high
order at the central point will push other zeros away.


\subsection{Correctness tests}

The method described here is simple enough that it is easy to implement, which reduces the
likeliness of bugs. It is still important to test it where possible, however, in order to
have more confidence in its correctness.

As described in Section \ref{section-small-conductor}, this code was run on every isogeny
class up to conductor $180000$, and
the results there suggest a high degree of confidence in the results elsewhere. As a further test,
one can also compute many zeros for the $L$-function of an elliptic curve of small conductor,
compute the sum over zeros directly, and verify that it agrees with our explicit formula
implementation. This was done with the elliptic curve ``11a1'' for a few values of $\Delta$,
and little over $200000$ zeros (computed using M. Rubinstein's lcalc package \cite{lcalc}),
and the agreement is generally to within about $10^{-6}$, which
is in line with what is expected using only $200000$ zeros, and which is roughly the precision
to which the integral in the explicit formula was calculated. Similar tests have also been done
with a smaller number of zeros for other $L$-functions.
\bibliography{/share/math/papers/bib}{}
\bibliographystyle{amsplain}

\end{document}